\documentclass[a4paper]{article}

\usepackage{
amsmath,
amsthm,
amscd,
amssymb,
}
\usepackage{stmaryrd}
\usepackage{comment}
\usepackage{tikz}
\usetikzlibrary{positioning,arrows,calc}
\usepackage{xspace}

\usepackage{graphicx}

\usepackage{bm}
\usepackage{url}

\theoremstyle{plain}
\newtheorem{lemma}{Lemma}

\newtheorem{proposition}{Proposition}
\newtheorem{theorem}{Theorem}
\newtheorem{conjecture}{Conjecture}
\newtheorem{question}{Question}

\makeatletter
\newtheorem*{rep@theorem}{\rep@title}
\newcommand{\newreptheorem}[2]{%
\newenvironment{rep#1}[1]{%
 \def\rep@title{#2 \ref{##1}}%
 \begin{rep@theorem}}%
 {\end{rep@theorem}}}
\makeatother

\newreptheorem{theorem}{Theorem}
\newreptheorem{lemma}{Lemma}
\newreptheorem{conjecture}{Conjecture}

\newtheoremstyle{derp}
{3pt}
{3pt}
{}
{}
{\upshape}
{:}
{.5em}
{}
\theoremstyle{derp}

\newcommand{\Z}{\mathbb{Z}}
\newcommand{\C}{\mathbb{C}}
\newcommand{\N}{\mathbb{N}}

\newcommand{\Aut}{\mathrm{Aut}}

\newcommand{\LAut}{\mathrm{LAut}}
\newcommand{\LEnd}{\mathrm{LEnd}}

\newcommand{\GL}{\mathrm{GL}}

\newcommand{\cwr}{\rbag}
\newcommand{\rot}{{\circlearrowleft}}

\newcommand{\sm}[1]{\left[\begin{smallmatrix}#1\end{smallmatrix}\right]}

\title{No Tits alternative for cellular automata}

\author{Ville Salo\thanks{email: \texttt{vosalo@utu.fi}}\\
  \multicolumn{1}{p{.7\textwidth}}{\centering\emph{University of Turku}}}

\begin{document}

\maketitle

\begin{abstract}
We show that the automorphism group of a one-dimensional full shift (the group of reversible cellular automata) does not satisfy the Tits alternative. That is, we construct a finitely-generated subgroup which is not virtually solvable yet does not contain a free group on two generators. We give constructions both in the two-sided case (spatially acting group $\Z$) and the one-sided case (spatially acting monoid $\N$, alphabet size at least eight). Lack of Tits alternative follows for several groups of symbolic (dynamical) origin: automorphism groups of two-sided one-dimensional uncountable sofic shifts, automorphism groups of multidimensional subshifts of finite type with positive entropy and dense minimal points, automorphism groups of full shifts over non-periodic groups, and the mapping class groups of two-sided one-dimensional transitive SFTs. We also show that the classical Tits alternative applies to one-dimensional (multi-track) reversible linear cellular automata over a finite field.
\end{abstract}

\section{Introduction}

In \cite{Ti72} Jacques Tits proved that if $F$ is a field (with no restrictions on characteristic), then a finitely-generated subgroup of $GL(n, F)$ either contains a free group on two generators or contains a solvable subgroup of finite index. We say that a group $G$ \emph{satisfies the Tits alternative} if whenever $H$ is a finitely-generated subgroup of $G$, either $H$ is virtually solvable or $H$ contains a free group with two generators. Whether an infinite group satisfies the Tits alternative is one of the natural questions to ask.

The fact that $GL(n, F)$ satisfies the Tits alternative implies several things:
\begin{itemize}
\item The `Von Neumann conjecture', that a group is amenable if and only if it contains no nonabelian free subgroup, is true for linear groups.\footnote{Mentioned open by Day in \cite{Da57}.}
\item Linear groups cannot have intermediate growth (between polynomial and exponential). Generally known as the \emph{Milnor problem} \cite{Mi68}.
\item Linear groups have no infinite finitely-generated periodic\footnote{A group $G$ is \emph{periodic}, or \emph{torsion}, if $\forall g \in G: \exists n \geq 1: g^n = 1$.} subgroups. Generally known as the \emph{Burnside problem} \cite{Bu02}.
\end{itemize}

The first item is true because solvable groups are amenable. The second is true by the theorem of Milnor \cite{Mi68} and Wolf \cite{Wo68}, which states that if $G$ is finitely-generated and solvable then either $G$ is virtually nilpotent or $G$ has exponential growth rate. The third is true because free groups are not periodic, and solvable groups cannot have finitely-generated infinite periodic subgroups because the group property ``all periodic f.g. subgroups are finite'' is satisfied by abelian groups and is closed under group extensions. 

These three properties (or lack thereof) are of much interest in group theory, since in each case whether groups can have these `pathological properties' was open for a long time. It seems that none of the three have been answered for automorphism groups of full shifts $\Aut(\Sigma^\Z)$ (equivalently, groups of reversible cellular automata).

\begin{question}
Does $\Aut(\Sigma^\Z)$ have a finitely-generated subgroup which is non-amenable, but contains no non-abelian free group?
\end{question}

\begin{question}
Does $\Aut(\Sigma^\Z)$ have a f.g. subgroup with intermediate growth?
\end{question}

\begin{question}
Does $\Aut(\Sigma^\Z)$ have an infinite f.g. subgroup which is periodic?
\end{question}

We show that the classical Tits alternative is not enough to solve the three questions listed -- it is not true. Concretely, we show that there is a residually finite variant of $A_5 \wr \Z$ which does not satisfy the Tits alternative and embeds in the automorphism group of a full shift.

A \emph{(two-sided) full shift} is $\Sigma^\Z$ where $\Sigma$ is a finite alphabet, with dynamics of $\Z$ given by the shift $\sigma : \Sigma^\Z \to \Sigma^\Z$ defined by $\sigma(x)_i = x_{i+1}$. A \emph{subshift} is a topologically closed shift-invariant subsystem of a full shift. A special case is a \emph{sofic shift}, a subsystem of a full shift obtained by forbidding a regular language of words, and \emph{SFTs} (subshift of finite type) are obtained by forbidding a finite language. 
An \emph{endomorphism} of a subshift is a continuous self-map of it, which commutes with the shift. The \emph{automorphism group} of a subshift $X$, denoted by $\Aut(X)$, is the group of endomorphisms having (left and right) inverses, and the automorphism group of a full shift is also known as the \emph{group of reversible cellular automata}. See \cite{LiMa95} for precise definitions, \cite{Ho10} for definitions in the multidimensional case, and \cite{CeCo10} for subshifts on general groups. All these notions have \emph{one-sided} versions where $\N$ is used in place of $\Z$. In the case of one-sided subshifts we will only discuss full shifts $\Sigma^\N$.

In symbolic dynamics, automorphism groups of subshifts are a classical topic \cite{He69,Co71}, with lots of progress in the 80's and 90's \cite{BoKr87,BoLiRu88,KiRo90,BoFi91} especially in sofic settings, but also in others \cite{HoPa89,FiFi96}. In the last few (half a dozen) years there has been a lot of interest in these groups 
\cite{Ol13,SaTo13a,Sa14a,CyKr14,CoYa14,CyKr15,DoDuMaPe16,DoDuMaPe17,CyFrKrPe16}
especially in settings where the automorphism group is, for one reason or another, more restricted. Also the full shift/sofic setting, which we concentrate on in this paper, has been studied in recent works \cite{FrScTa15,Sa15a,Sa16b,BaKaSa16}.

The popular opinion is that the automorphism group of a full shift is a complicated and intractable object. However, with the Tits alternative (or the three questions listed above) in mind, looking at known (published) finitely-generated subgroups as purely abstract groups does not really support this belief. As far as the author knows, all that is known about the set of finitely-generated subgroups of $\Aut(\Sigma^\Z)$ for a nontrivial alphabet $\Sigma$ follows from the facts that it is independent of the alphabet, contains the right-angled Artin groups (graph groups) and is closed under direct and free products (`cograph products') and finite group extensions (and of course contains the trivial group). See \cite{BoLiRu88,KiRo90,Sa15a}. All groups generated by these facts satisfy the Tits alternative by results of \cite{AnMi11}, see Proposition~\ref{prop:TitsClosures}.

Some of the known (families of) groups which satisfy the Tits alternative are hyperbolic groups \cite{Gr87}, outer automorphism groups of free groups \cite{BeFeHa00}, finitely-generated Coxeter groups \cite{NoVi02}, and right-angled Artin groups and more generally groups obtained by graph products from other groups satisfying the Tits alternative \cite{AnMi11}. In particular, we obtain that the automorphism group of a full shift contains a finitely-generated subgroup which is not embeddable in any such group.

Two particularly famous concrete examples of groups that do not satisfy the Tits alternative are the Grigorchuk group \cite{Gr80} and Thompson's group $F$ \cite{CaFlPa96}. These groups also have many other interesting properties, so it would be more interesting to embed them instead of inventing a new group for the task. The Grigorchuk group can indeed be embedded in the automorphism group of a subshift, by adapting the construction in \cite{Bo15a}, but the author does not know whether it embeds in the automorphism group of a sofic shift. Thompson's group $F$ embeds in the automorphism group of an SFT \cite{SaSc16a}, but is not residually finite, and thus does not embed in the automorphism group of a full shift. We mention also that there are solvable groups of derived length three whose automorphism groups do not satisfy the Tits alternative \cite{Ha76}.

The variant of $A_5 \wr \Z$ we describe is not a very complex group, and it is plausible that some weaker variant of the Tits alternative holds in automorphism groups of full shifts, and allows this type of groups in place of `virtually solvable groups'. In particular, our group is elementarily amenable \cite{Ch80}, and one could ask whether every finitely-generated subgroup of the automorphism group of a mixing SFT is either elementarily amenable or contains a free nonabelian subgroup. If this were the case, it would solve the Von Neumann, Milnor and Burnside problems for automorphism groups of mixing SFTs.

The group we construct satisfies the law $[g, h]^{30}$, and is thus an example of a residually finite group which satisfies a law, but does not satisfy the Tits alternative. It turns out that such an example has been found previously \cite[Theorem~1]{CoMa07}, and we were delighted to find that indeed our example sits precisely in the variety used in their theorem. However, our example is rather based on an answer\footnote{The answer reads `$A_5 \wr \Z$'.} of Ian Agol on the mathoverflow website \cite{Ia10}. The idea behind the embedding is based on Turing machines \cite{BaKaSa16} in the two-sided case. In the one-sided case we use a commutator trick applied to subshifts acting on finite sets.

\section{Results and corollaries}

In the two-sided case, we obtain several results, all based on the same construction (Lemma~\ref{lem:Construction}) and the fact the automorphism group of the full shift embeds quite universally into automorphism groups of subshifts.


\begin{theorem}
\label{thm:IntroMain}
For any finite alphabet $\Sigma$ with $|\Sigma| \geq 2$, the group $\Aut(\Sigma^\Z)$ of reversible cellular automata on the alphabet $\Sigma$ does not satisfy the Tits alternative.
\end{theorem}

The following summarizes the (well-known) embeddings listed in Section~\ref{sec:Embeddings}, and the corollaries for the Tits alternative.

\begin{reptheorem}{thm:Embeddings}
Let $A, \Sigma$ be finite alphabets, $|A| \geq 2$. Let $G$ be an infinite finitely-generated group, and $X \subset \Sigma^G$ a subshift. Then we have $\Aut(A^\Z) \leq \Aut(X)$, and thus $\Aut(X)$ does not satisfy the Tits alternative, if one of the following holds:
\begin{itemize}
\item $G = \Z$ and $X$ is an uncountable sofic shift,
\item $G = \Z^d$ and $X$ is a subshift of finite type with positive entropy and dense minimal points.
\item $G$ is not periodic and $X$ is a nontrivial full shift.
\end{itemize}
\end{reptheorem}

The first embedding is from \cite{KiRo90,Sa15a}, the second is from \cite{Ho10}. The third item is straightforward to prove, but suggests some interesting generalizations (see the discussion after the proof). 

The \emph{mapping class group} $\mathcal{M}_X$ of a subshift $X$ is defined in \cite{BoCh17}. Combining the embedding theorem of \cite{KiRo90,Sa15a} and results of \cite{BoCh17} (and a short additional argument) gives the following.

\begin{theorem}
\label{thm:Main4}
Let $A, \Sigma$ be finite alphabets, $|A| \geq 2$. If $X \subset \Sigma^\Z$ is a nontrivial transitive SFT, then $\Aut(A^\Z) \leq \mathcal{M}_X$, and thus $\mathcal{M}_X$ does not satisfy the Tits alternative.
\end{theorem}

Automorphisms of two-sided full shifts also appear in some less obviously symbolic dynamical contexts, in particular the \emph{rational group} $\mathcal{Q}$ of \cite{GrNeSu00} contains the automorphism group of a full shift, \cite{Wa90} constructs its classifying space (more generally those of mixing SFTs), implementing it as the fundamental group of a simplicial complex built out of integer matrices, and \cite{Wa86} `realizes'\footnote{While the precise statement in \cite{Wa86} is natural and interesting, it seems hard to get a non-trivial group-theoretic interepretation of this result just from the statement of the theorem: it does follow that the automorphism group of a full shift is a subquotient of the homeomorphism group of the sphere but since this group contains free groups, so is any other countable group.} automorphisms of full shifts in the centralizer of a particular homeomorphism of a sphere of any dimension at least~$5$. 

In the case of one-sided SFTs, there are nontrivial restrictions on finite subgroups, and the group is generated by elements of finite order. The automorphism group of $\Sigma^\N$ is not finitely-generated if $|\Sigma| \geq 3$, while $|\Aut(\{0,1\})^\N| = 2$ \cite{He69,BoFrKi90}. We prove that the Tits alternative also fails in the one-sided case.

\begin{reptheorem}{thm:OneSided}
Let $|\Sigma| \geq 8$. Then $\Aut(\Sigma^\N)$ does not satisfy the Tits alternative.
\end{reptheorem}

This group embeds in $\Aut(\Sigma^\Z)$, so Theorem~\ref{thm:IntroMain} follows from Theorem~\ref{thm:OneSided}. While $\Aut(\Sigma^\Z)$ embeds in many situations where we have a symbolic group action, $\Aut(\Sigma^\N)$ embeds more generally in many monoid and (non-vertex-transitive) graph contexts, though unlike in the case of $\Z$, we are not aware of anything concrete to cite.

The group $\Aut(\Sigma^\N)$ also arises in a less obviously symbolic context: It is shown in \cite{BlDeKe91} that if $S_d$ denotes the set of centered (coefficient of $x^{d-1}$ is zero) monic polynomials $p$ of degree $d$ over the complex numbers, such that the filled-in Julia set (points whose iteration stays bounded) of $p$ does not contain any critical point, then $\pi_1(S_d)$ (the fundamental group of $S_d$ as a subspace of $\C^d$) admits $\Aut(\{0, \ldots, d-1\}^\N)$ as a quotient. Unfortunately, the Tits alternative is not closed under quotients (since free groups satisfy the Tits alternative trivially by the Nielsen-Schreier theorem), so we do not directly get a corollary for $\pi_1(S_d)$.





\section{Residually finite wreath product}

Grigorchuk group and Thompson's group $F$ are presumably particularly famous examples of groups not satisfying the Tits alternative mostly because they are particularly famous for other reasons, and happen not to satisfy the Tits alternative -- one can construct such examples directly by group extensions: $A_5 \wr \Z$ does not satisfy the Tits alternative by a similar proof as that of Lemma~\ref{lem:NoTits}.

The group $A_5 \wr \Z$ is not residually finite since $A_5$ is non-abelian \cite[Theorem~3.1]{Gr57}, and thus cannot be embedded in the automorphism group of a full shift. Informally, there is an obvious way to `embed' it, but this embedding is not quite a homomorphism, because the relations only hold on some `well-formed' configurations. In this section, we describe the abstract group obtained through this `embedding' -- a kind of broken wreath product. Luckily for us, it turns out not to satisfy the Tits alternative either.

Let $N \subset \N$ be an infinite\footnote{The definition works fine if $N$ is a finite set, but then the group we define is finite, and thus satisfies the Tits alternative.} set, let $G$ be a finite group and write $G \cwr^N \Z$ for the group generated by the elements of $G$ and a new element $\rot$, which act on $\bigsqcup_{n \in N} G^n$ by
\[ a \cdot (g_1, g_2, \ldots, g_n) = (a g_1, g_2, \ldots, g_n) \]
for $a \in G$ and $n \in N$, and
\[ \rot \cdot (g_1, g_2, \ldots, g_n) = (g_2, g_3, \ldots, g_n, g_1). \]
More precisely, the formulas attach a bijection on $\bigsqcup_{n \in N} G^n$ to each $a \in G$ and to $\rot$ (by the above formulas), and $G \cwr^N \Z$ is the group of bijections they generate. This is a variant of the usual wreath product of $G$ and $\Z$, but $G \cwr^N \Z$ is obviously residually finite for any finite group $G$, since it is defined by its action on the finite sets $G^n$. Note that $\rot$ simply rotates (the coordinates of) $G^n$ for $n \in N$, and generates a copy of $\Z$.

A \emph{subquotient} of a group is a quotient of a subgroup. 

\begin{lemma}
\label{lem:Subquotient}
Let $H$ be a finitely-generated group which has $A_5^n$ as a subquotient for infinitely many $n$. Then $H$ is not virtually solvable.
\end{lemma}

\begin{proof}
Suppose it is. Then there is an index $k$ solvable subgroup $K \leq H$. Let $H_1 \leq H$ be a subgroup. 
Let $h : H_1 \to H_2$ be a surjective homomorphism. Then for $K' = h(K \cap H_1)$, we have $[H_2 : K'] \leq [H_1 : K \cap H_1] \leq k$. Because $K'$ is a subquotient of the solvable group $K$, it is solvable, and we obtain that any subquotient of $H$ has a solvable subgroup of index at most $k$.

The group $A_5^n$ is now a subquotient of $H$ for arbitrarily large (thus all) $n$, so $A_5^n$ has a solvable subgroup $K'$ of index at most $k$. Since $K'$ is solvable, its projection to any coordinate $i$ of the product $A_5^n$ must be a proper subgroup of $A_5$, thus of index at least $5$. Thus we have $[A_5^n : K'] \geq 5^n > k$ for large enough $n$, a contradiction.
\end{proof}

\begin{lemma}
\label{lem:NoTits}
The group $A_5 \cwr^N \Z$ is not virtually solvable, but satisfies a law and thus does not contain a free nonabelian subgroup. 
\end{lemma}

\begin{proof}
First observe that $\phi(a) = 0$ for $a \in G$ and $\phi(\rot) = 1$ extends to a well-defined homomorphism $\phi : A_5 \cwr^N \Z \to \Z$ (where we use that $N$ is infinite).

Suppose that $g, h \leq A_5 \cwr^N \Z$ and $N \subset \N$. Commutators vanish under any homomorphism to an abelian group, so $\phi([g, h]) = 0$. Thus independently of $n$, we have $[g, h]^{30} \cdot \vec v = \vec v$ for any $n$ and $\vec v \in A_5^n$, since the exponent of $A_5$ is $30$. This implies that $g$ and $h$ satisfy a nontrivial relation. Since $g, h$ were arbitrary, the group satisfies the law $[g, h]^{30}$. It follows that no two elements generate a free subgroup. 

So show that $A_5 \cwr^N \Z$ is not virtually solvable, we show that it has $A_5^n$ as a subquotient for arbitrarily large $n$. We show it is a subgroup of a quotient (thus also a quotient of a subgroup). Pick $n \in N$. Then $A_5 \cwr^N \Z$ acts on $A_5^n$ and a moment's reflection shows that this induces a surjective homomorphism from $A_5 \cwr^N \Z$ to $A_5 \wr \Z/n\Z$. We have $A_5^n \leq A_5 \wr \Z/n\Z$, so $A_5^n$ is a subquotient and we conclude with Lemma~\ref{lem:Subquotient}.
\end{proof}

Since $A_5$ acts faithfully on $\{1,2,3,4,5\}$, it is easy to see that the following group (again defined by its action) is isomorphic to $A_5 \cwr^N \Z$: Let elements of $A_5$ and $\rot$ act on $\bigsqcup_{n \in N} \{1,2,3,4,5\}^n$ by
\[ a \cdot (m_1, m_2, \ldots, m_n) = (a \cdot m_1, m_2, \ldots, m_n) \]
for $a \in A_5$, and
\[ \rot \cdot (m_1, m_2, \ldots, m_n) = (m_2, m_3, \ldots, m_n, m_1). \]

We do not study the structure of $A_5 \cwr^N \Z$ in detail, but make a few observations. First, this group surjects onto the classical wreath product $A_5 \wr \Z$ by observing that any identity between the $\rot$ and $a \in A_5$ as generators of $A_5 \cwr^N \Z$ in particular hold for their action on $A_5^n$ for arbitrarily large $n \in N$. If $n$ is sufficiently large (larger than the length of the identity, to ensure information does not have time to travel ``around the circle''), then the identity must hold in $A_5 \wr \Z$.

Second, while $A_5 \cwr^N \Z$ is not virtually solvable, it is (locally finite)-by-abelian (just like $A_5 \wr \Z$), showing that it sits rather low in the elementary amenable hierarchy.

\section{The construction in the two-sided case}

If $\Sigma$ is a finite alphabet, write $\Sigma^*$ for the set of words over $\Sigma$ (including the empty word).

\begin{lemma}
\label{lem:Construction}
There exists an alphabet $\Sigma$ such that we have $A_5 \cwr^{2\N} \Z \leq \Aut(\Sigma^\Z)$.
\end{lemma}

\begin{proof}
Let $A = \{1,2,3,4,5\}$ and choose $\Sigma = \{\#\} \cup A^2$. Before describing the automorphism, we define an auxiliary faithful action of $A_5 \cwr^{2\N} \Z$ on finite words. Think of a word $w \in (A^2)^n$ as two words $u, v \in A^n$ on top of each other, the topmost one defined by $u_i = (w_i)_1$ and the second $v_i = (w_i)_2$, for $i = 1, 2, \ldots, n$. We use the notation $w = \sm{u \\ v} = (\sm{u_1 \\ v_1}, \sm{u_2 \\ v_2}, \ldots, \sm{u_n \\ v_n})$. Define a bijection $\psi : (A^2)^n \to A^{2n}$ by $\psi(\sm{u \\ v}) = uv^R$ where $v^R$ is the reversal of $v$ defined by $a^R = a$ for $a \in A$ and $(va)^R = a(v^R)$.

Now, conjugate the defining action of $A_5 \cwr^{2\N} \Z$ on $A^{2n}$ to $(A^2)^n$ through $\psi$ to obtain the action
\[ a \cdot (\sm{m_1 \\ m_1'}, \sm{m_2 \\ m_2'}, \ldots, \sm{m_n \\ m_n'}) = (\sm{a \cdot m_1 \\ m_1'}, \sm{m_2 \\ m_2'}, \ldots, \sm{m_n \\ m_n'}) \]
for $a \in A_5$, and for $\rot$ the following counter-clockwise `conveyor belt' rotation
\[ \rot \cdot (\sm{m_1 \\ m_1'}, \sm{m_2 \\ m_2'}, \ldots, \sm{m_n \\ m_n'}) = (\sm{m_2 \\ m_1}, \sm{m_3 \\ m_1'}, \sm{m_4 \\ m_2'}, \ldots, \sm{m_n' \\ m_{n-1}'}). \]

Now, we define our automorphisms. To $a \in A_5$ we associate the automorphism $f_a : \Sigma^\Z \to \Sigma^\Z$ defined by $f_a(x)_i = F_a(x_{i-1}, x_i)$ where $F_a : \Sigma^2 \to \Sigma$ is defined by $F_a(b, c) = c$ if $b \neq \#$, $F_a(\#, \#) = \#$ and
\[ F_a(\#, \sm{b \\ c}) = \sm{a \cdot b \\ c} \]
where $a \cdot b$, is the action of $a \in A_5$ by permutation on $b \in A$. It is easy to see that $f_a$ is an endomorphism of $\Sigma^\Z$, and $x_i = \# \iff f_a(x)_i = \#$.

To $\rot$, we associate $f_{\rot} : \Sigma^\Z \to \Sigma^\Z$ defined by $f_{\rot}(x)_i = F_\rot(x_{i-1}, x_i, x_{i+1})_i$ where $F_\rot : \Sigma^3 \to \Sigma$ is defined by $F_\rot(a, \#, b) = \#$ for all $a, b \in \Sigma$ and
\[ F_\rot(\#, \sm{c \\ d}, \#) = \sm{d \\ c}, \]
\[ F_\rot(\#, \sm{c \\ d}, \sm{e \\ f}) = \sm{e \\ c}, \]
\[ F_\rot(\sm{a \\ b}, \sm{c \\ d}, \#) = \sm{d \\ b}, \]
\[ F_\rot(\sm{a \\ b}, \sm{c \\ d}, \sm{e \\ f}) = \sm{e \\ b} \]
for all $a,b,c,d,e,f \in A$. It is easy to see that $F_\rot$ is also an endomorphism of $\Sigma^\Z$, and $x_i = \# \iff F_\rot(x)_i = \#$.

Now, let $Y \subset \Sigma^\Z$ be the set of points $x$ where both the left tail $x_{(-\infty,-1]}$ and the right tail $x_{[0,\infty)}$ contain infinitely many $\#$-symbols, and consider any point $x \in Y$. Then $x$ splits uniquely into an infinite concatenation of words
\[ x = \ldots w_{-2} \# w_{-1} \# w_0 \# w_1 \# w_2 \ldots \]
where $w_i \in (A^2)^*$ for all $i \in \Z$. If $f$ is either $f_{\rot}$ or one of the $f_a$ for $a \in A_5$, then the decomposition of $f(x)$ contains $\#$s in the same positions as that of $x$, in the sense that (up to shifting indices)
\[ f(x) = \ldots u_{-2} \# u_{-1} \# u_0 \# u_1 \# u_2 \ldots \]
where $|u_i| = |w_i|$ for all $i$ and the words begin in the same coordinates. Thus $f(x) \in Y$. It is easy to see that between two $\#$s, the mapping $w_i \mapsto u_i$ performed by $f$ is precisely the one we defined for words in $(A^2)^*$ described above for the corresponding generator of $A_5 \cwr^{2\N} \Z$.

It follows that $a \mapsto f_a|_Y$, $\rot \mapsto f_{\rot}|_Y$ extends uniquely to an embedding of $A_5 \cwr^{2\N} \Z$ into the group of self-homeomorphisms of $Y$. Since $Y$ is dense in $\Sigma^\Z$ and $f_a$ and $f_{\rot}$ are endomorphisms of $\Sigma^\Z$, $a \mapsto f_a$, $\rot \mapsto f_{\rot}$ extends to an embedding of $A_5 \cwr^{2\N} \Z$ into $\Aut(\Sigma^\Z)$.
\end{proof}

There are natural interpretations of the action on the limit points $x \in \Sigma^\Z \setminus Y$. On a configuration $x \in \Sigma^\Z$ where $x_i = \#$ and $x_{j} \neq \#$ for all $j > i$, $f_a$ and $f_{\rot}$ simulate the usual wreath product $A_5 \wr \Z$ on the right tail $x_{[i+1, \infty)}$, and a similar claim holds for left tails. This yields the homomorphism to $A_5 \wr \Z$ mentioned in the previous section. On configurations where $\#$ does not appear at all, the action is by shifting the top and bottom tracks, and gives the $\phi$-homomorphism used in Lemma~\ref{lem:NoTits}. 

\section{Embedding results}
\label{sec:Embeddings}

In this section we list some embeddings from the literature. We start with uncountable sofic shifts, where uncountable refers to the cardinality of the set of points. Note that full shifts are uncountable sofic shifts ($\Sigma^\N$ is uncountable, and the empty language is regular).

The following is \cite[Lemma~7]{Sa15a}.

\begin{lemma}
\label{lem:USSEmbedding}
If $X \subset \Sigma^\Z$ is an uncountable sofic shift, then $\Aut(A^\Z) \leq \Aut(X)$ for any finite alphabet $A$.
\end{lemma}

\begin{proposition}
\label{prop:Embedding}
If $X$ is an uncountable sofic shift, then $A_5 \cwr^{\N} \Z \leq \Aut(X)$.
\end{proposition}

\begin{proof}
By Lemma~\ref{lem:USSEmbedding}, we have $\Aut(\Sigma^\Z) \leq \Aut(X)$ where $\Sigma$ is the alphabet of Lemma~\ref{lem:Construction}. We have $A_5 \cwr^{2\N} \Z \leq \Aut(\Sigma^\Z)$ by Lemma~\ref{lem:Construction}, so it is enough to check that $A_5 \cwr^{\N} \Z \leq A_5 \cwr^{2\N} \Z$. One can check that such an embedding is induced by $a \mapsto a$ for $a \in A_5$, and $\rot \mapsto \rot^2$.
\end{proof}



As for countable sofic shifts, we do not have a characterization of the situations when the automorphism group satisfies the Tits alternative. However, in that setting, there are stronger methods for studying the three embeddability questions listed in the introduction and we refer to \cite{SaSc16a}.

The following is \cite[Theorem~3]{Ho10}.

\begin{lemma}
\label{lem:Hochman}
If $X \subset \Sigma^{\Z^d}$ is an SFT with positive entropy and dense minimal points, then $\Aut(A^\Z) \leq \Aut(X)$ for any finite alphabet $A$.
\end{lemma}

Below, minimal points are points whose orbit-closure is minimal as a dynamical system.

\begin{theorem}
\label{thm:Embeddings}
Let $A, \Sigma$ be finite alphabets, $|A| \geq 2$. Let $G$ be an infinite finitely-generated group, and $X \subset \Sigma^G$ a subshift. Then we have $\Aut(A^\Z) \leq \Aut(X)$, and thus $\Aut(X)$ does not satisfy the Tits alternative, if one of the following holds:
\begin{itemize}
\item $G = \Z$ and $X$ is an uncountable sofic shift,
\item $G = \Z^d$ and $X$ is a subshift of finite type with positive entropy and dense minimal points.
\item $G$ is not periodic and $X$ is a nontrivial full shift.
\end{itemize}
\end{theorem}

\begin{proof}
The first two were proved above. The third item is true, because if $\Z \leq G$ by $n \mapsto g^n$ and $X = A^G$ for $A \subset \Sigma$, then if $K$ is a set of left cosets representatives for $\langle g \rangle$ then $f \in \Aut(A^\Z)$ directly turns into an automorphism of $A^G$ by $\hat f(x)_{kg^n} = f(y)_n$ where $y_i = x_{kg^i}$ and $k \in K$.
\end{proof}

The first item generalizes to sofic $H \times \Z$-shifts where $H$ is finite by \cite{Sa16b}, and could presumably be generalized to virtually-cyclic groups with the same idea. By symmetry-breaking arguments, we believe the third item generalizes to SFTs with a nontrivial point of finite support, on any group which is not locally finite, and also to cellular automata acting on sets of colorings of undirected graphs, but this is beyond the scope of the present paper. A generalization of the second item seems worth conjecturing more explicitly.

\begin{conjecture}
Let $G$ be an amenable group which is not locally finite and $X$ a subshift of finite type with positive entropy and dense minimal points. Then we have $\Aut(A^\Z) \leq \Aut(X)$ for any finite alphabet $A$.
\end{conjecture}

Next, we deal with the mapping class group. By definition, the group $\langle \sigma \rangle$ is contained in the center of $\Aut(X)$ for any $\Z$-subshift $X$, in particular it is a normal subgroup.

\begin{lemma}
Let $X$ be any uncountable sofic shift. Then $\Aut(A^\Z) \leq \Aut(X) / \langle \sigma \rangle$ for every finite alphabet $A$.
\end{lemma}

\begin{proof}
Let $\phi : \Aut(A^\Z) \to \Aut(X)$ be the embedding given by Lemma~\ref{lem:USSEmbedding}. From its proof in \cite{Sa15a}, it is easy to see that there exists an infinite subshift $Y \leq X$ such that $\phi(f)$ fixes every point in $Y$ for every $f \in \Aut(A^\Z)$; namely the maps $\phi(f)$ only act nontrivially at a bounded distance from an unbordered word $w$ which can be taken to be arbitrarily long.

We show that based on only this, $\phi$ is automatically also an embedding of $\Aut(A^\Z)$ into $\Aut(X) / \langle \sigma \rangle$. Suppose not, and that $\phi(f) \circ \sigma^k = \phi(g)$ for some $f, g \in \Aut(A^\Z)$. Then in particular $\phi(f) \circ \sigma^k(y) = \phi(g)(y) \implies \sigma^k(y) = y$ for every $y \in Y$. If $k \neq 0$ this is a contradiction since $Y$ is an infinite subshift. If $k = 0$, then $\phi(f) = \phi(g)$ implies $f = g$ since $\phi : \Aut(A^\Z) \to \Aut(X)$ is an embedding.
\end{proof}

The following is now a straightforward corollary of \cite{BoCh17}. See \cite{BoCh17} for the definition of the \emph{mapping class group} $\mathcal{M}_X$ of a subshift $X$ (which is not needed in the proof).

\begin{lemma}
\label{lem:MappingClassEmbedding}
Let $X$ be any transitive SFT. Then $\Aut(A^\Z) \leq \mathcal{M}_X$ for every alphabet $A$.
\end{lemma}

\begin{proof}
In \cite[Theorem~5.6]{BoCh17}, it is shown in particular that if $X$ is a transitive SFT, then its mapping class group contains an isomorphic copy of $\Aut(X) / \langle \sigma \rangle$, which then contains a copy of $\Aut(A^\Z)$ by the above lemma.
\end{proof}


In \cite[Corollary~5.5]{GrNeSu00}, it is shown that the automorphism group of every full shift embeds in the group $\mathcal{Q}$ that they define, sometimes called the \emph{rational group}. Thus we also obtain a new proof that $\mathcal{Q}$ does not satisfy the Tits alternative.

\section{One-sided automorphism groups}

The automorphism group of the full shift $\{0,1\}^\N$ is isomorphic to $\Z/2\Z$. For large enough alphabets, however, we show that $\Aut(\Sigma^\N)$ does not satisfy the Tits alternative. This gives another proof of Theorem~\ref{thm:IntroMain}. 

The high-level idea of the proof is that we can associate to a subshift a kind of action of it on a finite set, in a natural way. Mapping $n \mapsto (0^{n-1} 1)^\Z$, the action of the subshift generated by the image of $N \subset \N$ corresponds to the group $A_5 \cwr^N \Z$ defined previously. It turns out that Lemma~\ref{lem:NoTits} generalizes to such actions, and any infinite subshift can be used in place of this (almost) periodic subshift. The generalization is based on a commutator trick from \cite{Ba89,Ma06,AaGrSc15,BoKaSa16,BaKaSa16}. The trick to adapting the construction to cellular automata on $\N$ is to consider `actions of the trace of another cellular automaton'.


We only give the definitions in the special case of $A_5$. Let $X \subset A^\Z$ be a subshift and let $Y = \{1,2,3,4,5\}$. For each $g \in A_5$ and $j \in A$ define a bijection $g_j : Y \times X \to Y \times X$ by $g_j(y, x) = (g \cdot y, x)$ for $x_0 = j$ and $g_j(y, x) = (y, x)$ for $x_0 \neq j$. Define a bijection $\rot \curvearrowright Y \times X$ by $\rot(y, x) = (y, \sigma(x))$. Denote the group generated by these maps by $A_5 \cwr^X \Z$.

\begin{lemma}
Let $X \subset A^\Z$ be infinite. Then the group $A_5 \cwr^X \Z$ is not virtually solvable, but satisfies a law and thus does not contain a free nonabelian group.
\end{lemma}

\begin{proof}
Observe that $\phi(g) = 0$ for $g \in A_5$ and $\phi(\rot) = 1$ extends to a well-defined homomorphism $\phi : A_5 \cwr^X \Z \to \Z$ (since $X$ is infinite). Suppose first that $g, h \in A_5 \cwr^N \Z$ are arbitrary. Then as in Lemma~\ref{lem:NoTits}, we get $\phi([g, h]) = 0$, and again we have $[g, h]^{30} \cdot (y, x) = (y, x)$.

We now show that $A_5^n$ is a subgroup of $A_5 \cwr^X \Z$ for arbitrarily large $n$, after which the claim follows from Lemma~\ref{lem:Subquotient}. 

Consider the cylinder sets $[w]_i = \{x \in A^\Z \;|\; x_{[i,i+|w|-1]} = w\}$ and for $g \in A_5$, $w \in \{0,1\}^*$ and $i \in \Z$ define
\[ \pi_{g,i,w}(y, x) = (g y, x) \mbox{ if } x \in [w]_i, \mbox{ and } \pi_{g,i,w}(y, x) = (y, x) \mbox{ otherwise}. \]

We claim that $\pi_{g,i,w} \in A_5 \cwr^X \Z$ for all $g, i, w$. To see this, observe that by the definition of how $g_j \in A_5$ acts on $Y \times X$ for $j \in A$, and by conjugating with a power of $\rot$, we have $\pi_{g,w,i} \in A_5 \cwr^X \Z$ for all $g \in A_5, i \in \Z$ and $w \in A$. We proceed inductively on $|w|$: Let $w = uv$ where $u, v \in A^+$ are nonempty words, and let $a \in A_5$ be any commutator, that is, $a = [b, c]$. Then one sees easily that
\[ [\pi_{b,i,u}, \pi_{c,i+|u|,v}] = \pi_{[b, c],i,w}. \]
Because $A_5$ is perfect (that is, $A_5 = [A_5, A_5]$), we get that $\pi_{g, i, w} \in A_5 \cwr^X \Z$ for every $g \in A_5$. This proves the claim.

Now, let $w$ be an unbordered word\footnote{Alternatively, one can use the Marker Lemma \cite{LiMa95} for $X$.} of length $n$ \cite{Lo02} that occurs in some point of $X$. Then the elements $\pi_{g,i,w}$ for $0 \leq i < |w|$ generate a group isomorphic to $A_5^n$ (because the fact $w$ is unbordered implies the supports of their actions are disjoint), and we conclude.
\end{proof}


\begin{theorem}
\label{thm:OneSided}
Let $|\Sigma| \geq 8$. Then $\Aut(\Sigma^\N)$ does not satisfy the Tits alternative.
\end{theorem}

\begin{proof}
Let $\Sigma = Y \sqcup A$ where $A$ is a finite set, $|A| \geq 3$ and $|Y| = 5$.

Let $f : \Sigma^\N \to \Sigma^\N$ be any reversible cellular automaton of infinite order such that $x_i \in A \iff f(x)_i \in A$, $x_i \notin A \implies f(x)_i = x_i$, and the values in $Y$ do not affect the behavior of $f$, in the sense that if $x, y \in \Sigma^\Z$ and $\forall i: (x_i \notin A \wedge y_i \notin A) \vee x_i = y_i$, then $\forall i: x_i \in A \implies f(x)_i = f(y_i)$.

One construction is the following: for $a, b \in A$, $a \neq b$, define $F_{a, b} : \Sigma^2 \to \Sigma$ by
\[ \mbox{for all } c \notin \{a, b\}: F_{a,b}(a,c) = b \mbox{ and } F_{a, b}(b, c) = a \]
and $F_{a, b}(c, d) = c$ otherwise. Define $f_{a,b} \in \Aut(\Sigma^\N)$ by $f_{a,b}(x)_0 = F_{a,b}(x_0,x_1)$. This is an involution, so it is an automorphism. Now it is easy to check that for any distinct elements $a, b, c \in A$, $f_{a, b} \circ f_{b, c}$ is of infinite order, by considering its action on any point that begins with the word $b^ka$.

The \emph{trace} of $f$ is the subshift of $\Sigma^\Z$ consisting of $y$ such that for some $x \in \Sigma^\N$, $f^n(x)_0 = y_n$ for all $n$. Since $f$ is of infinite order on $\Sigma^\N$ and fixes all symbols not in $A$, its trace $T_f$ intersected with $A^\Z$ is infinite. Write $X = T_f \cap A^\Z$ for this infinite subshift. 

For each permutation $p \in A_5$ and $j \in A$, take the $f_{p,j}$ to be the automorphism defined by $f_{p,j}(x)_0 = p(x_0)$ if $x_0 \in Y$, $x_1 \in A$ and $\pi(x_1) = j$, and by $f_{p,j}(x)_0 = x_0$ otherwise. Let $F$ denote the group generated by $f$ and $f_{p,j}$ for $p \in A_5$, $j \in A$.

The group $F$ is isomorphic to $A_5 \cwr^X \Z$ by the isomorphism $f \mapsto \rot$, $f_{p,j} \mapsto p_j$. To see this, we show that the same identities are satisfied by the generators. First, suppose some word $w$ over the generators $\rot$ and $p_i$ acts as the identity in $A_5 \cwr^X \Z$, and consider any point $x \in \Sigma^\Z$. Since $X$ is infinite, the total rotation $\phi(w)$ is zero. Thus when $w$ is evaluated in $F$, $f$ is applied as many times as $f^{-1}$, so all coordinates $x_i \in A$ return to their original values (observe that values of coordinates containing a symbol from $Y$ do not affect coordinates with a symbol from $A$). Coordinates $i$ with $x_i \in Y$, $x_{i+1} \in Y$ are not affected by $w$. If $x_i \in Y$, $x_{i+1} \in A$, then the action of $f$ and $f_{p,j}$ exactly simulates the action of $\rot$ and $p_j$ on $(x_i, y)$ where $y \in X$ is the configuration defined by $y_n = f^n(x)_{i+1}$, since shifting a configuration of the trace corresponds to applying $f$.

Now, suppose $w$ does not represent the identity in $A_5 \cwr^X \Z$, and suppose $w$ does not fix $(a, z) \in Y \times X$. By the assumption that $X$ is the trace of $f$ intersected with $A^\Z$, we can find a configuration $x \in \Sigma^\Z$ where for some $i$, $f^n(x)_i = z_n$ for all $n$. By shift-commutation we may assume $i = 1$ and since $f$ and its inverse are one-sided we can assume $x_0 = a$. Then clearly $w$ does not act as identity on $x$.

By the previous lemma, $A_5 \cwr^X \Z$ does not satisfy the Tits alternative. Since $\Aut(\Sigma^\N)$ contains $F \cong A_5 \cwr^X \Z$ it does not satisfy the Tits alternative either.
\end{proof}

\begin{question}
For which one-sided transitive SFTs does the Tits alternative hold? Is it always false when the automorphism group is infinite?
\end{question}

Unlike in the two-sided case, automorphism groups of one-sided full shifts do not in general embed into each other, due to restrictions on finite subgroups, see \cite{BoFrKi90}.

\section{Tits alternative for linear cellular automata}

We show that linear cellular automata, with a suitable definition, satisfy the assumptions of Tits' theorem directly.

Let $K$ be a finite field\footnote{Finiteness is not really needed -- the arguments go through with any field, with any T1 topology on $K$ and induced topology on $V$, and the product topology on $V^\Z$.} and let $V$ be a finite-dimensional vector space over $K$. Then the full shift $V^\Z$ is also a vector space over $K$ with cellwise addition $(x + y)_i = x_i + y_i$ and cellwise scalar multiplication $(a \cdot x)_i = a \cdot x_i$. Consider the semigroup $\LEnd(V^\Z)$ of endomorphisms $f$ of the full shift $V^\Z$ which are also linear maps, i.e. $f(x + y) = f(x) + f(y)$, $a \cdot f(x) = f(a \cdot x)$, and the group $\LAut(V^\Z) \leq \LEnd(V^\Z)$ of such maps which are also automorphisms (in which case the automorphism is automatically linear as well).

The (formal) Laurent series $\sum_{i \in \Z} c_i \bm x^i$, where $c_i \in K$ for all $i$ and $c_i = 0$ for all small enough $i$, form a field that we denote by $K((\bm x))$. By $\GL(n, K((\bm x)))$ we denote the group of linear automorphisms of the vector space $K((\bm x))^n$ over $K((\bm x))$. A Laurent polynomial is a Laurent series with finitely many nonzero coefficients $c_i$, and we write $K[\bm x, \bm x^{-1}]$ for this subring of $K((\bm x))$.

It is standard that the semigroups $\LEnd(V^\Z)$ and $M_n(K[\bm x, \bm x^{-1}])$ are isomorphic (by collecting the images of points with support of cardinality $1$ into a matrix) and this isomorphism maps $\LAut(V^\Z)$ onto the set of invertible matrices. Since $K[\bm x, \bm x^{-1}]$ is a subring of $K((\bm x))$, it follows that $\LAut(V^\Z) \leq \GL(n, K((\bm x)))$ where $n$ is the dimension of $V$, and we get the following.

\begin{theorem}
Let $V$ be an $n$-dimensional vector space over a finite field $K$. Then the group $\LAut(V^\Z)$ embeds in $\GL(n, K((\bm x)))$, and thus satisfies the Tits alternative.
\end{theorem}

\begin{question}
Let $X \subset \Sigma^\Z$ be a group shift \cite{BoSc08,Ki87} (that is, $X$ supports a shift-commuting continuous group structure). Does the group $\LAut(X)$ of automorphisms respecting the group structure satisfy the Tits alternative?
\end{question}

\section{Tits alternative for some known subgroups}

As mentioned in the introduction, as far as the author knows, there are no previously published finitely-generated subgroups of $\Aut(\Sigma^\Z)$ which cannot be built by graph products and finite extensions from the trivial group (note that this family contains all finite groups, all f.g. free groups, all f.g. abelian groups and all graph groups). Let us show that all groups generated by these closure properties satisfy the Tits alternative (in the sense of the present paper). This is a direct corollary of results of \cite{AnMi11}.

\begin{proposition}
\label{prop:TitsClosures}
Let $\mathcal{G}$ be the family of finitely-generated groups satisfying the Tits alternative. Then $\mathcal{G}$ is closed under subgroups, finite graph products and finite group extensions.
\end{proposition}

\begin{proof}
The case of subgroups is trivial. Theorem A of \cite{AnMi11} directly implies that a finite graph product of groups satisfies the Tits alternative (in our sense) whenever the vertex groups do: our Tits alternative is ``Tits Alternative relative to $(\mathcal{I}_f, \mathcal{C}_{vsol})$'' in their terminology, and $(\mathcal{I}_f, \mathcal{C}_{vsol})$ satisfies their assumptions.

Now suppose $H \leq G$ is of finite index and $H$ satisfies the Tits alternative. Let $K \leq G$ be a finitely-generated subgroup. Since $H$ is of finite index, $K \cap H$ is finitely-generated. Since $H$ satisfies the Tits alternative, either the group $K \cap H$ contains a (nonabelian) free subgroup, or it is virtually solvable. If it contains a free group, so does $K$. If it is virtually solvable, then there is a solvable subgroup $F \leq K \cap H$ of finite index. Then we have
\[ [K : F] \leq [K : K \cap H] [K \cap H : F] \leq [G : H] [K \cap H : F] < \infty \]
so $K$ is virtually solvable.
\end{proof}

\section*{Acknowledgements}

The author thanks Ilkka T\"orm\"a for his comments on the paper, and especially for suggesting a much simpler embedding in the two-sided case. The author thanks Johan Kopra for his comments on the article. The last section arose from discussions with Pierre Guillon and Guillaume Theyssier. The question of Tits alternative for cellular automata was proposed by Tom Meyerovitch. The author thanks the referee for their comments.

\bibliographystyle{plain}
\bibliography{../../../bib/bib}{}

\end{document}